%% file: main.tex
\documentclass{article}

\usepackage{amsmath,amssymb,amsthm}
\usepackage{graphicx}
\usepackage{bm}

\usepackage{hyperref}
\usepackage{color}

\theoremstyle{definition}
\newtheorem{definition}{Definition}[section]

\theoremstyle{plain}
\newtheorem{theorem}[definition]{Theorem}
\newtheorem{proposition}[definition]{Proposition}
\newtheorem{lemma}[definition]{Lemma}

\theoremstyle{remark}

\title{Recursive Patterns in the Chocolate Game}
\author{
  Tomoro Okubo \and
  Yuzuri Kashiwagi \and
  Nobumitsu Niida
}
\date{}

\begin{document}

\maketitle

\begin{abstract}
We study the recursive structure of P-positions in the chocolate game \(C_{m,m}\), an impartial game played on an \(m \times m\) chocolate bar. We show that the set of P-positions exhibits self-similar patterns that can be described and enumerated recursively. We further establish a correspondence between these patterns and the cross-sections of a three-dimensional Sierpiński octahedron. Finally, we show that the P-positions can be generated by a second-order cellular automaton, analogous to the one-dimensional Rule-60 automaton. Our results reveal deep connections between combinatorial games, fractal geometry, and discrete dynamical systems.
\end{abstract}

\section{Introduction}
This paper explores the structural properties of $P$-positions in the chocolate game $C_{m,n}$, an impartial game played on a rectangular bar. While Nim-like games are typically studied through algebraic means, we show that $C_{m,m}$ admits a rich geometric and computational description.

Specifically, we prove that the $P$-positions of $C_{m,m}$ exhibit a self-similar pattern equivalent to cross-sections of the Sierpi\'nski octahedron. Furthermore, we show that this pattern emerges from a second-order cellular automaton related to Wolfram’s Rule-60. By connecting game theory with fractal geometry and cellular automata, we provide both a recursive enumeration and a novel visualization of winning strategies. 

The structure of the paper is as follows: Section 2 reviews the basics of the chocolate game. 
Sections 3, 4, and 5 present our main results on P-positions, self-similarity, and cellular automata, 
followed by a discussion of future directions.

\section{Chocolate Game and P-Positions}

\begin{definition}
Let $m,n$ be positive integers.
The \emph{chocolate game} $C_{m,n}$ is a two-player impartial game defined as follows.

The game is played on an $m \times n$ rectangular chocolate bar divided into unit square cells,
exactly one of which is poisoned.
The players alternately perform the following move:
\begin{itemize}
  \item Break the chocolate bar along grid lines into two rectangular pieces,
        eat one piece, and give the other piece to the opponent.
  \item If the chocolate bar cannot be broken, the player must eat it.
\end{itemize}
The player who eats the poisoned cell loses.
The position of the poisoned cell is assumed to be visible to both players.

We assume that a player always avoids eating the poisoned cell whenever a legal move allows this.
The terminal position of the game is defined to be the position in which a $1 \times 1$ chocolate bar
is passed to a player and no further move is possible.
\end{definition}

\begin{definition}
For the game $C_{m,n}$, define
\[
B_{m,n}=\{(i,j)\in \mathbb{Z}^2 \mid 1\le i\le m,\ 1\le j\le n\},
\]
and call $B_{m,n}$ the \emph{board} of $C_{m,n}$.
Elements of $B_{m,n}$ are called \emph{cells}.

A cell corresponding to a poisoned position for which the second player has a winning strategy
is called a \emph{P-position}.
Let $P_{m,n}\subset B_{m,n}$ denote the set of all P-positions.

We define the \emph{pattern} of $P_{m,n}$ as the set
\[
R_{m,n}
=
\bigl\{(i-x,j-y)\in\mathbb{R}^2
\mid 0\le x\le1,\ 0\le y\le1,\ (i,j)\in P_{m,n}\bigr\},
\]
which gives a geometric visualization of $P_{m,n}$ (see Figure~\ref{fig:pattern}).

For a cell $(i,j)\in B_{m,n}$, we define its \emph{cell value} by
\[
(i-1)\oplus(j-1)\oplus(m-i)\oplus(n-j),
\]
where $\oplus$ denotes the bitwise exclusive OR (xor) of integers.

In the square case $m=n$, define
\[
v_m(i,j)=
\begin{cases}
0, & \text{if } (i-1)\oplus(j-1)\oplus(m-i)\oplus(m-j)=0,\\
1, & \text{otherwise}.
\end{cases}
\]

Finally, for a subset $A\subset\mathbb{R}^2$ and a real number $r$, we write
\[
rA=\{(rx,ry)\mid(x,y)\in A\}.
\]
\end{definition}

\begin{figure}[htbp]
\centering
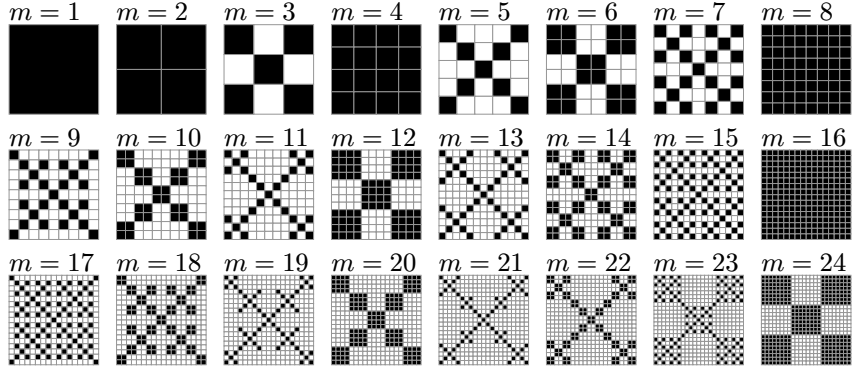
\caption{The scaled pattern $\frac{1}{m}R_{m,m}$}
\label{fig:pattern}
\end{figure}

The following theorem is a well-known criterion for P-positions in the game of
four-pile Nim.
A proof can be found, for example, in \cite{Sato2014}.
Here we state and prove the result in the context of the square chocolate game $C_{m,m}$.

\begin{theorem}
\label{thm:nim}
For the chocolate game $C_{m,m}$, a cell $(i,j)$ is a P-position if and only if
\[
(i-1)\oplus(j-1)\oplus(m-i)\oplus(m-j)=0.
\]
\end{theorem}

\begin{proof}
Let $(i,j)$ be the cell corresponding to the poisoned position.
A single move of breaking the chocolate can be regarded as choosing one of the
four nonzero integers
\[
i-1,\quad j-1,\quad m-i,\quad m-j,
\]
and decreasing it by a positive integer so that it remains nonnegative.

By symmetry, it suffices to consider the case in which $i-1$ is decreased.
Set
\[
X=(i-1)\oplus(j-1)\oplus(m-i)\oplus(m-j).
\]
Suppose that one move transforms the position into
\[
Y=(i-\alpha)\oplus(j-1)\oplus(m-i)\oplus(m-j),
\qquad 2\le \alpha\le i.
\]
Using the associativity and commutativity of $\oplus$ together with
$(i-1)\oplus(i-1)=0$, we obtain
\[
Y=X\oplus(i-1)\oplus(i-\alpha).
\]

We claim that the following two statements hold:
\begin{enumerate}
\item If $X=0$, then $Y\neq0$ for any legal move.
\item If $X\neq0$, then there exists a legal move such that $Y=0$.
\end{enumerate}

For (1), note that $(i-1)\oplus(i-\alpha)\neq0$, which implies $Y\neq0$.

To prove (2), write $X$ in binary form and let $2^p$ be the highest power of $2$
appearing in $X$.
If the $2^p$-digit of $i-1$ is $1$, then
\[
X\oplus(i-1)<(i-1),
\]
and hence we can choose $\alpha$ so that
\[
i-\alpha = X\oplus(i-1),
\]
which yields $Y=0$.
The same argument applies if the $2^p$-digit is $1$ in $j-1$, $m-i$, or $m-j$.

The terminal position is a P-position and satisfies $X=0$. Since conditions (1) and (2) characterize the recursive structure of the game, it follows that a position is a P-position if and only if $X=0$. 
\end{proof}

\section{Recursive Structure of P-Positions}

\begin{proposition}
\label{prop:symmetry}
The pattern $R_{m,m}$ is symmetric with respect to the horizontal and vertical
axes, and also with respect to the diagonal.
More precisely, for any cell $(i,j)$, the cells
\[
(i,j),\ (i,m-j+1),\ (m-i+1,j),\ (j,i)
\]
are either all P-positions or all non-P-positions.
\end{proposition}

\begin{proof}
The cell values of all four cells coincide and are equal to
\[
(i-1)\oplus(j-1)\oplus(m-i)\oplus(m-j).
\]
The claim therefore follows from Theorem~\ref{thm:nim}.
\end{proof}

\begin{proposition}
\label{prop:diagonal}
In the chocolate game $C_{m,m}$, all cells on the two diagonals are P-positions.
That is, for any integer $1\le k\le m$, the cells $(k,k)$ and $(k,m+1-k)$ are
P-positions.
\end{proposition}

\begin{proof}
For these cells, the corresponding cell value satisfies
\[
(k-1)\oplus(k-1)\oplus(m-k)\oplus(m-k)=0.
\]
The assertion follows again from Theorem~\ref{thm:nim}.
\end{proof}

\begin{theorem}
\label{thm:doubling}
For any positive integer $m$,
\[
R_{2m,2m}=2R_{m,m}.
\]
In particular, $R_{m,m}$ and $R_{2m,2m}$ are similar figures.
\end{theorem}

\begin{proof}
It suffices to show that
\[
v_{2m}(2i,2j)
=
v_{2m}(2i-1,2j)
=
v_{2m}(2i,2j-1)
=
v_{2m}(2i-1,2j-1)
=
v_m(i,j)
\]
for all $1\le i,j\le m$.

Using binary representations, the exclusive OR operation satisfies
\begin{equation}
\label{eq:xor}
\begin{aligned}
2a\oplus2b &= 2(a\oplus b),\\
(2a+1)\oplus(2b+1) &= 2(a\oplus b).\nonumber
\end{aligned}
\end{equation}

We compute
\begin{align*}
&(2i-2)\oplus(2j-2)\oplus(2m-2i+1)\oplus(2m-2j+1)\\
&=2(i-1)\oplus2(j-1)\oplus\{2(m-i)+1\}\oplus\{2(m-j)+1\}\\
&=2\bigl((i-1)\oplus(j-1)\oplus(m-i)\oplus(m-j)\bigr),
\end{align*}
which implies $v_{2m}(2i-1,2j-1)=v_m(i,j)$.

Similarly,
\begin{align*}
&(2i-1)\oplus(2j-1)\oplus(2m-2i)\oplus(2m-2j)\\
&=\{2(i-1)+1\}\oplus\{2(j-1)+1\}\oplus2(m-i)\oplus2(m-j)\\
&=2\bigl((i-1)\oplus(j-1)\oplus(m-i)\oplus(m-j)\bigr),
\end{align*}
so that $v_{2m}(2i,2j)=v_m(i,j)$.

The remaining two cases,
$v_{2m}(2i-1,2j)$ and $v_{2m}(2i,2j-1)$, are treated in the same way.
Hence the claim follows.
\end{proof}

\begin{theorem}
\label{thm:decomposition}
Let $m$ be a positive integer and write its binary expansion as
\[
m = a_0 + a_1 2 + a_2 2^2 + \cdots + a_k 2^k + 2^{k+1}.
\]
Define
\[
s = a_0 + a_1 2 + a_2 2^2 + \cdots + a_k 2^k,
\qquad
t = 2^{k+1} - s.
\]
Then the pattern $R_{m,m}$ consists of four copies of $R_{s,s}$ placed at the
corners and one copy of $R_{t,t}$ placed at the center.
\end{theorem}

For example, if $m=11=1011_{2}$, then $s=011_{2}=3$ and $t=101_{2}=5$.
In this case, $R_{11,11}$ consists of four copies of $R_{3,3}$ at the corners
and one copy of $R_{5,5}$ in the center (see Figure~\ref{fig:decomposition}).

\begin{figure}[htbp]
\centering
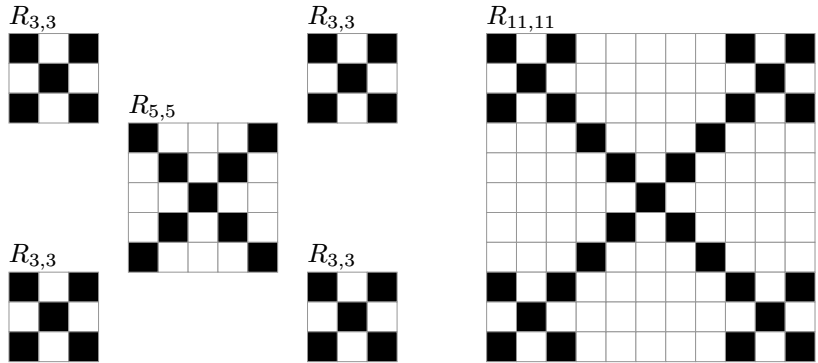
\caption{$R_{3,3}$, $R_{5,5}$, and $R_{11,11}$}
\label{fig:decomposition}
\end{figure}

\begin{proof}
By symmetry (Proposition~\ref{prop:symmetry}), it suffices to verify the following
three cases.

\begin{enumerate}
\item
If $1\le i\le s$ and $1\le j\le s$, then since $m=2^{k+1}+s$, we have
\begin{align*}
v_m(i,j)
&=(i-1)\oplus(j-1)\oplus(2^{k+1}+s-i)\oplus(2^{k+1}+s-j)\\
&=(i-1)\oplus(j-1)\oplus(s-i)\oplus(s-j)\\
&=v_s(i,j).
\end{align*}

\item
If $s+1\le i\le 2^{k+1}$ and $s+1\le j\le 2^{k+1}$, then
\begin{align*}
v_m(i,j)
&=(i-1)\oplus(j-1)\oplus(2s+t-i)\oplus(2s+t-j)\\
&=v_t(i-s,j-s).
\end{align*}

\item
If $s+1\le i\le 2^{k+1}$ and $1\le j\le s$, then
\[
v_m(i,j) = (i-1)\oplus(j-1)\oplus
(2^{k+1}-(i-s))\oplus(2^{k+1}+s-j).
\]

Since $i-s \geq 1$ and $s-j \geq 0$, notice that the terms $2^{k+1}-(i-s)$ and $2^{k+1}+s-j$ straddle a power of two. In particular,
\[
2^{k+1}-(i-s) < 2^{k+1}, \quad 2^{k+1}+s-j \geq 2^{k+1}.
\]

This implies that the exclusive OR of these four terms cannot vanish, i.e.,
\[
v_m(i,j) = (i-1)\oplus(2^{k+1}-i+s)\oplus(j-1)\oplus(2^{k+1}+s-j) \neq 0.
\]
\end{enumerate}

Combining these cases proves the theorem.
\end{proof}

Propositions~\ref{prop:symmetry} and \ref{prop:diagonal}, together with
Theorems~\ref{thm:doubling} and \ref{thm:decomposition}, show that
the set $P_{m,m}$ can be determined recursively.
\section{Enumeration of P-Positions}

In this section, we investigate the number of P-positions.
Let
\[
g(m)=\# P_{m,m}
\]
denote the number of P-positions for the square chocolate game $C_{m,m}$.
For a positive integer $x$, let $u(x)$ denote the largest nonnegative integer
such that $2^{u(x)}$ divides $x$.

\begin{lemma}
\label{lem:recurrence}
For any positive integer $m$, the following relations hold:
\begin{align*}
g(2m) &= 4g(m),\\
g(2m+1) &= g(m)+g(m+1).
\end{align*}
\end{lemma}

\begin{proof}
The identity $g(2m)=4g(m)$ follows immediately from
Theorem~\ref{thm:doubling}.

We prove $g(2m+1)=g(m)+g(m+1)$.
It suffices to show that
\[
v_{2m+1}(2i-1,2j-1)=v_{m+1}(i,j),\quad
v_{2m+1}(2i,2j)=v_m(i,j),
\]
and
\[
v_{2m+1}(2i-1,2j)=v_{2m+1}(2i,2j-1)=1.
\]

First,
\begin{align*}
&(2i-2)\oplus(2j-2)\oplus(2m+1-2i+1)\oplus(2m+1-2j+1)\\
&=2\bigl((i-1)\oplus(j-1)\oplus(m+1-i)\oplus(m+1-j)\bigr),
\end{align*}
which implies $v_{2m+1}(2i-1,2j-1)=v_{m+1}(i,j)$.

Next,
\begin{align*}
&(2i-1)\oplus(2j-1)\oplus(2m+1-2i)\oplus(2m+1-2j)\\
&=\{2(i-1)+1\}\oplus\{2(j-1)+1\}\oplus\{2(m-i)+1\}\oplus\{2(m-j)+1\}\\
&=2\bigl((i-1)\oplus(j-1)\oplus(m-i)\oplus(m-j)\bigr),
\end{align*}
so $v_{2m+1}(2i,2j)=v_m(i,j)$.

Finally,
\begin{align*}
&(2i-2)\oplus(2j-1)\oplus(2m+1-2i+1)\oplus(2m+1-2j)\\
&=2(i-1)\oplus\{2(j-1)+1\}\oplus2(m+1-i)\oplus\{2(m-j)+1\}\\
&=2\bigl( (i-1)\oplus(j-1)\oplus(m+1-i)\oplus(m-j)\bigr) .
\end{align*}
If $m$ is even, then the least significant bit of
$(i-1)\oplus(m+1-i)$ is $0$, while that of
$(j-1)\oplus(m-j)$ is $1$.
If $m$ is odd, the situation is reversed.
In either case, the resulting xor is nonzero, and hence
$v_{2m+1}(2i-1,2j)=1$.
By the same argument,
$v_{2m+1}(2i,2j-1)=1$.

Therefore $g(2m+1)=g(m)+g(m+1)$, completing the proof.
\end{proof}

\begin{theorem}
\label{thm:odd-sum}
For any positive integer $n$,
\[
\sum_{m=1}^{2^{\,n-1}} g(2m-1)=6^{\,n-1}.
\]
\end{theorem}

\begin{proof}
We prove the claim by induction on $n$.

For $n=1$, we have $g(1)=1=6^0$, so the statement holds.

Assume that the assertion holds for all integers up to $n=k$.
For integers $x$ with $1\le x<2^{k+1}$, the number of integers satisfying
$u(x)=y$ is $2^{k-y}$.
Using Lemma~\ref{lem:recurrence}, we obtain
\begin{align*}
\sum_{m=1}^{2^{k}-1} g(m)
&=\sum_{i=1}^{k}\sum_{j=1}^{2^{i-1}} g\bigl(2^{k-i}(2j-1)\bigr)\\
&=\sum_{i=1}^{k}\sum_{j=1}^{2^{i-1}}4^{k-i}g(2j-1)\\
&=\sum_{i=1}^{k}4^{k-i}6^{i-1}\\
&=4^{k-1}\sum_{i=1}^{k}\left(\frac{3}{2}\right)^{i-1}\\
&=2^{2k-1}\left(\left(\frac{3}{2}\right)^k-1\right)\\
&=2^{k-1}(3^k-2^k).
\end{align*}

Therefore,
\begin{align*}
\sum_{m=1}^{2^{k}} g(2m-1)
&=g(1)+\sum_{m=1}^{2^{k}-1} g(2m+1)\\
&=g(1)+\sum_{m=1}^{2^{k}-1}\bigl(g(m)+g(m+1)\bigr)\\
&=2\sum_{m=1}^{2^{k}-1} g(m)+g(2^k)\\
&=2^{k}(3^k-2^k)+2^{2k}\\
&=6^{k}.
\end{align*}
This completes the induction.
\end{proof}

\begin{theorem}
\label{thm:total-sum}
For any positive integer $n$,
\[
\sum_{m=1}^{2^{n}} g(m)=\frac{4^{n}+6^{n}}{2}.
\]
\end{theorem}

\begin{proof}
By Theorem~\ref{thm:odd-sum},
\begin{align*}
\sum_{m=1}^{2^{n}} g(m)
&=\sum_{m=1}^{2^{n}-1} g(m)+g(2^{n})\\
&=2^{n-1}(3^{n}-2^{n})+2^{2n}\\
&=\frac{4^{n}+6^{n}}{2}.
\end{align*}
\end{proof}
\section{Self-Similar Structure and Cellular Automata}

By integrating the results obtained so far, we observe that the patterns
appearing in the chocolate game are similar to planar sections of a certain
three-dimensional self-similar structure, which converges to a fractal in the
limit defined below.

\begin{definition}[Sierpiński Octahedron of order $n$]
    
Let us define a geometric operation $T$ on a regular octahedron as follows:

\begin{enumerate}
    \item Let $O$ be the center of the octahedron.
    \item For each face of the octahedron, consider the tetrahedron formed by the three midpoints of the edges of the face and the point $O$.
    \item Remove all six such tetrahedra from the original octahedron.
\end{enumerate}

Applying operation $T$ to a regular octahedron yields six smaller octahedra, each with half the edge length(see Figure~\ref{fig:operation}).

\begin{figure}
\centering
\includegraphics[width=60mm]{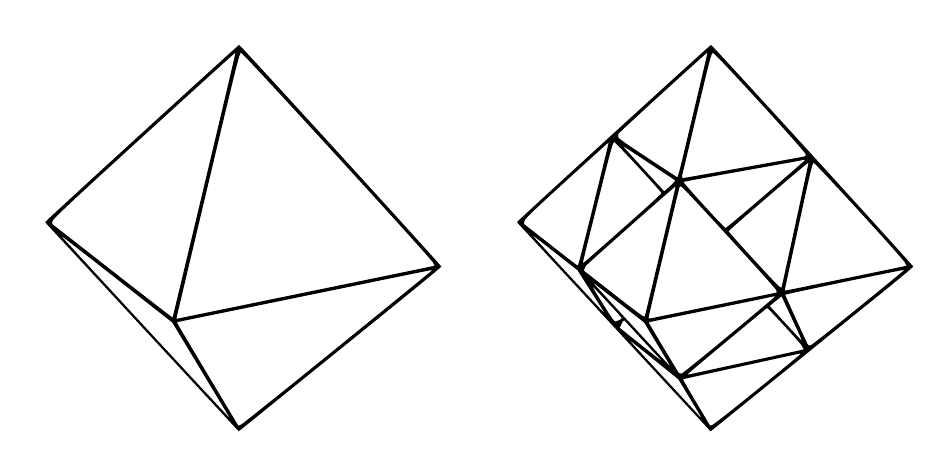}
\caption{Operation
$T$ on a regular octahedron }
\label{fig:operation}
\end{figure}

Starting from the regular octahedron with vertices
\[
(\pm1,0,0),\ (0,\pm1,0),\ (0,0,\pm1),
\]
we iteratively apply the operation $T$.  
For each $n \ge 0$, we denote by $S_n$ the resulting figure consisting of $6^n$ regular octahedra of edge length $\frac{\sqrt{2}}{2^n}$, and call $S_n$ the \emph{Sierpiński octahedron of order $n$}.

For integers $m$, we define the \emph{$m$-th horizontal section} of $S_n$ by
\[
H_{n,m} := S_n \cap \left\{ z = 1 - \frac{m}{2^n} \right\}.
\]
\end{definition}

\begin{theorem}
For $0 < m \le 2^n$, the pattern $R_{m,m}$ is similar to the section $H_{n,m}$.
\end{theorem}

Figure~\ref{fig:sierpinski-section} illustrates the $11$-th section $H_{4,11}$ of the Sierpiński octahedron $S_4$, which is similar to $R_{11,11}$.

\begin{figure}
\centering
\includegraphics[width=60mm]{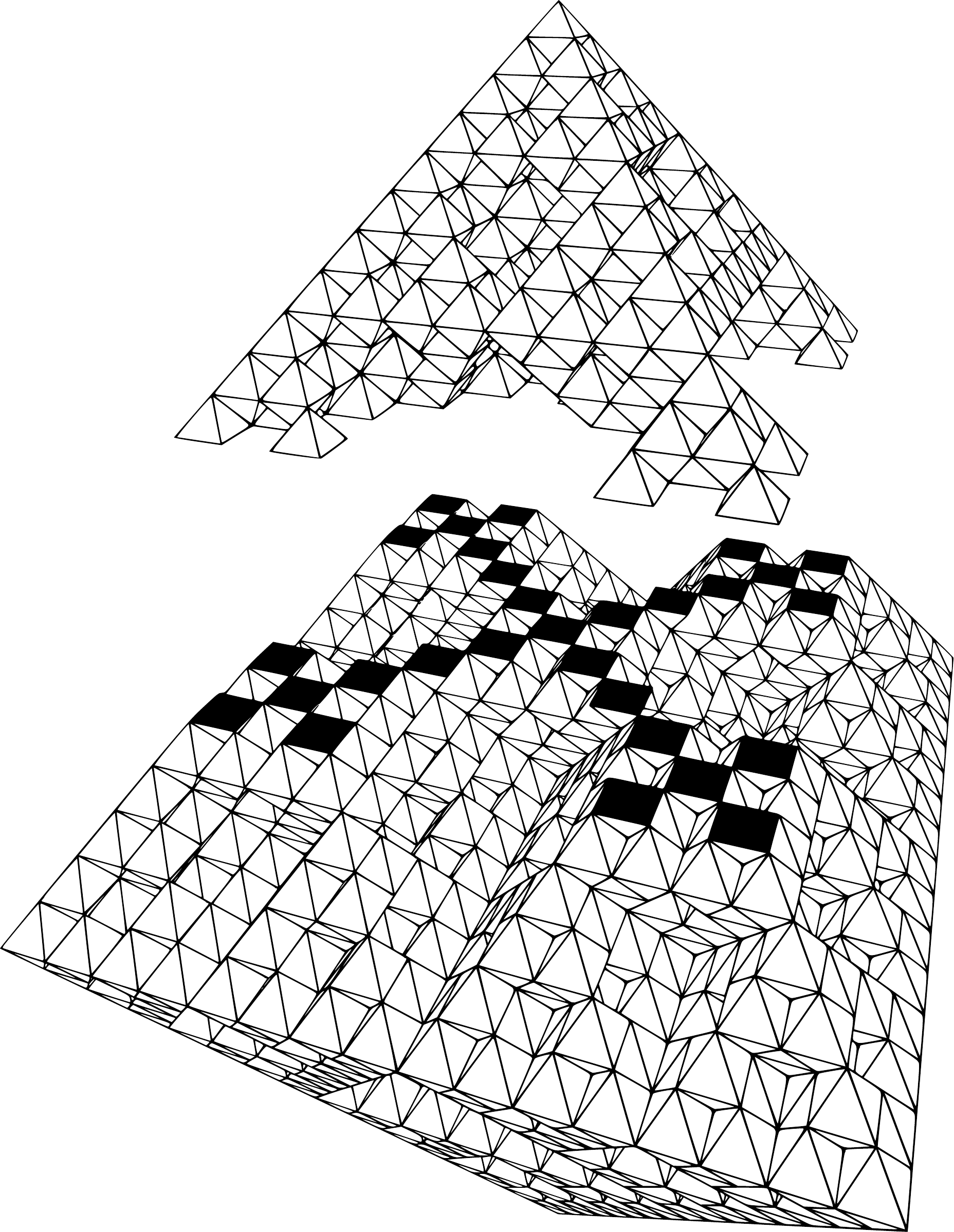}
\caption{A horizontal section of the Sierpiński octahedron of order 4}
\label{fig:sierpinski-section}
\end{figure}

\begin{proof}
Let $s$ be the integer obtained by removing the highest-order bit in the binary expansion of $m$, and set $t = m - 2s$.
By construction of $S_n$, the section $H_{n,m}$ consists of four copies of $H_{n,s}$ located at the corners and one copy of $H_{n,t}$ located at the center.

For $0 \le k \le n$, the section $H_{n,2^k}$ is congruent to
\[
\frac{1}{2^{n-1}} R_{2^k,2^k}.
\]
Therefore, the claim follows from Theorem~\ref{thm:decomposition}.
\end{proof}

Using this property, we give an alternative proof of Theorem~\ref{thm:total-sum}.

\begin{proof}[Alternative proof of Theorem~\ref{thm:total-sum}]
Since
\[
g(2^n) + 2 \sum_{m=1}^{2^n-1} g(m) = 6^n
\quad \text{and} \quad
g(2^n) = 2^{2n} = 4^n,
\]
we obtain
\[
\sum_{m=1}^{2^n} g(m)
= \frac{6^n - 4^n}{2} + 4^n
= \frac{6^n + 4^n}{2}.
\]
\end{proof}

Next, we consider sections of $S_n$ at half-integer levels.
For a positive integer $m$ with $m<2^n$, let $H_{n,m+\frac{1}{2}}$ denote the intersection of $S_n$ with the plane
\[
z = 1 - \frac{m+\frac{1}{2}}{2^n}.
\]
By the argument used in Lemma~\ref{lem:recurrence}, it follows that $H_{n,m+\frac{1}{2}}$ is congruent to
$\frac{1}{2} H_{n,2m+1}$(see Figure~\ref{fig:decomposition2}).

\begin{figure}[htbp]
\centering
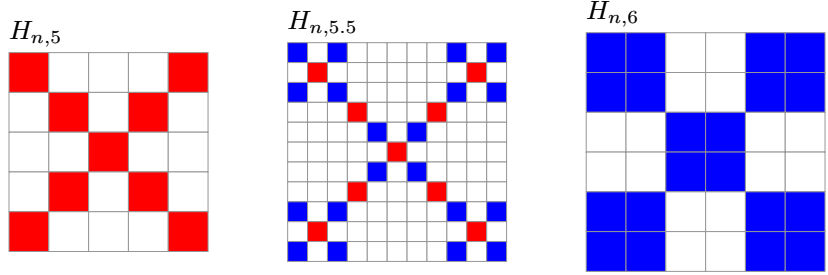
\caption{$H_{n,5}$, $H_{n,5.5}$, and $H_{n,6}$}
\label{fig:decomposition2}
\end{figure}

Furthermore, let $S_\infty$ denote the limit of $S_n$ as $n \to \infty$.
This limit is known as the \emph{Sierpiński octahedron}, a fractal object.
For a real number $0 \le r \le 1$, define $H_{r}$ to be the intersection of
$S_\infty$ with the plane $z = 1 - r$.
By the definition of the operation $T$, we have
\[
H_{n,m} = H_{ m/2^n}.
\]

In particular, since $m < 2^m$, the following theorem holds.
\begin{theorem}
For any positive integer $m$, the pattern $R_{m,m}$ is similar to the section
$H_{m/2^m}$ of $S_\infty$.
\end{theorem}

The set $P_{m,m}$ can also be generated by a second-order cellular automaton, which may be regarded as a two-dimensional analogue of Rule-60 in one-dimensional cellular automata.

\begin{theorem}
For integers $i,j$, define a sequence of integers $\{ a_{i,j}(n) \}_{n \ge 0}$ recursively by
\begin{align*}
a_{i,j}(n+2) &=
a_{i,j}(n+1) + a_{i-1,j}(n+1) + a_{i,j-1}(n+1) + a_{i-1,j-1}(n+1) \\
&\quad + a_{i-1,j-1}(n),
\end{align*}
with initial conditions
\[
a_{i,j}(0) = 0 \quad \text{for all } i,j,
\]
and
\[
a_{i,j}(1) =
\begin{cases}
1 & \text{if } i=j=1, \\
0 & \text{otherwise}.
\end{cases}
\]
Then
\[
P_{m,m}
=
\left\{ (i,j) \in \mathbb{Z}^2 \,\middle|\,
a_{i,j}(m) \equiv 1 \pmod{2}
\right\}.
\]
\end{theorem}

\begin{proof}
We prove by induction on $m$ that
\[
v_m(i,j) = 0
\quad \Longleftrightarrow \quad
a_{i,j}(m) \equiv 1 \pmod{2}.
\]

For $m=1$, we have $v_1(1,1)=0$, so the claim holds.

Assume that the statement holds for all integers up to $m=n$.
For positive integers $a,b$, the following relations are valid:
\[
a \oplus b =
\begin{cases}
(a-1)\oplus b + 1 & \text{if } u(a) < u(b), \\
a \oplus (b-1) + 1 & \text{if } u(a) > u(b), \\
(a-1)\oplus (b-1) & \text{if } u(a) = u(b).
\end{cases}
\]
In particular, the condition $u(a)=u(b)$ is equivalent to
\[
a \oplus (b-1) = (a-1)\oplus b.
\]

Suppose that
\[
v_{n+1}(i+1,j+1)
= i \oplus j \oplus (n-i) \oplus (n-j) = 0.
\]
Reorder $i,n-i$ as $a,b$ so that $u(a) \le u(b)$, and similarly reorder $j,n-j$ as $c,d$ with $u(c) \le u(d)$.
Then $a+b=c+d=n$ and $a\oplus b=c\oplus d$.

\medskip
\noindent
(1) If \(u(a) < u(b)\) and \(u(c) < u(d)\), then
\[
(a-1)\oplus b \oplus (c-1)\oplus d = 0.
\]
In this case, The following equivalence holds:
\[
a\oplus(b-1)=c\oplus(d-1)
\quad \Longleftrightarrow \quad
(a-1)\oplus(b-1)=(c-1)\oplus(d-1).
\]
Moreover, assuming \((a-1) \oplus b = c \oplus (d-1)\), it follows that
\((c-1) \oplus d = c \oplus (d-1)\), which contradicts \(u(c) < u(d)\).
Therefore,
\[
(a-1) \oplus b \oplus c \oplus (d-1) \neq 0.
\]
Similarly, we have
\[
a \oplus (b-1) \oplus (c-1) \oplus d \neq 0.
\]

\medskip
\noindent
(2) If $u(a)=u(b)$ or $u(c)=u(d)$, then necessarily
\[
u(a)=u(b)=u(c)=u(d),
\]
since the carry positions in the binary addition of $a$ and $b$ coincide with those in the binary addition of $c$ and $d$.
Thus,
\[
(a-1)\oplus(b-1)\oplus(c-1)\oplus(d-1)=0,
\]
and all four of the following expressions vanish:
\[
(a-1)\oplus b \oplus (c-1)\oplus d,
\quad
(a-1)\oplus b \oplus c \oplus (d-1),
\]
\[
a\oplus(b-1)\oplus(c-1)\oplus d,
\quad
a\oplus(b-1)\oplus c \oplus (d-1).
\]

By cases (1), (2), and the induction hypothesis, we conclude that
\[
v_{n+1}(i+1,j+1)=0
\quad \Longleftrightarrow \quad
a_{i+1,j+1}(n+1) \equiv 1 \pmod{2},
\]
which completes the proof.
\end{proof}

\section{A Relation of Our Presentation of P-positions and a Nim with a Pass}
In this article, we presented a different perspective on the chocolate problems, and in this section, we introduce interesting facts that support the usefulness of our perspective. These facts were provided by Dr. Ryohei Miyadera, and we included them in this article with his permission.

An interesting but challenging question in combinatorial game theory has been determining what happens when standard game rules are modified to allow a \textit{one-time pass}. This pass move may be used at most once in the game and not from the terminal position. Once either player has used a pass, it is no longer available. In the case of classical Nim, the introduction of the pass alters the game's mathematical structure, considerably increasing its complexity, and finding the formula that describes the set of previous players' winning positions remains an important open question that has defied traditional approaches. 

In \cite{nimpass2} and \cite{nimpassfinite}, Chan, Low, Locke, and Wong described the set of previous players' positions of Nim with a pass when the number of stones in each pile is at most four. This study shows how difficult it is to obtain the formula for the previous player's winning positions in classical Nim with a pass. 

The graph of the traditional three-pile Nim has an elegant mathematical structure. See  Figure \ref{nim}. However, the graph of the traditional three-pile Nim with a pass does not seem to have any definite structure. See Figure \ref{nimpass}. 

\begin{figure}[htbp]
\begin{tabular}{cc}
\begin{minipage}[t]{0.5\textwidth}
\begin{center}
\includegraphics[height=3cm]{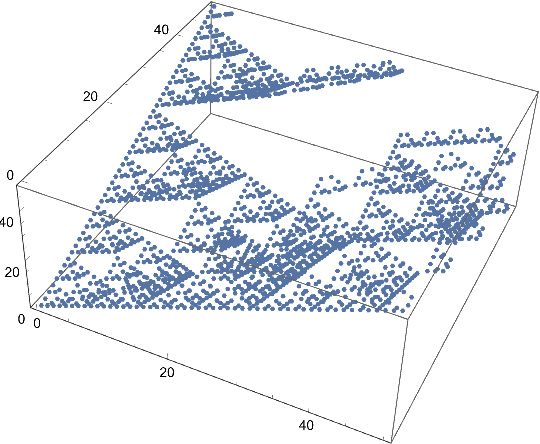}
\caption{ Graph of Nim}
\label{nim}
\end{center}
\end{minipage}
\begin{minipage}[t]{0.5\textwidth}
\begin{center}
\includegraphics[height=3cm]{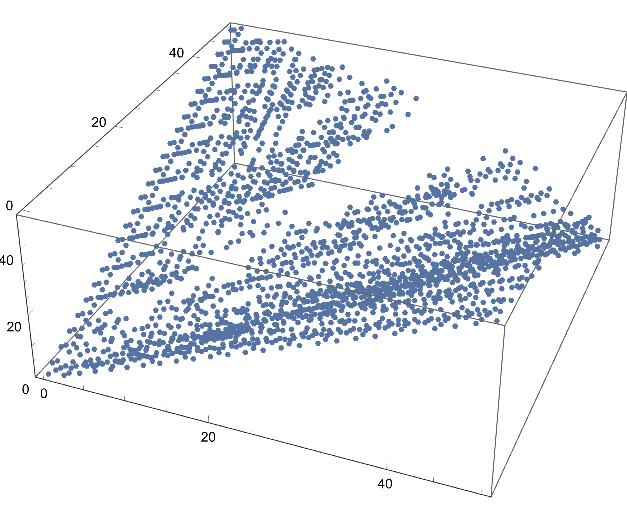}
\caption{Graph of Nim with a Pass }
\label{nimpass}
\end{center}
\end{minipage}
\end{tabular}
\end{figure}

Next, we examine the previous player's winning positions in four-pile Nim with a pass from the perspective introduced in the previous sections.
In Figures \ref{s14},  \ref{s18},  \ref{s21}, and  \ref{s22}, blue rectangles are the previous player's winning positions in Nim, and red rectangles are the previous player's winning positions in Nim with a pass.
The red rectangles have some definite mathematical structures. 
Therefore, the perspective introduced in the previous sections presented 
some definite mathematical structure in the history of the traditional Nim with more than two piles and a pass.
This supports the usefulness of our perspective on chocolate problems.

\begin{figure}[htbp]
\begin{tabular}{cc}
\begin{minipage}[t]{0.5\textwidth}
\begin{center}
\includegraphics[height=3cm]{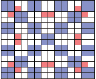}
\caption{ $m=14$}
\label{s14}
\end{center}
\end{minipage}
\begin{minipage}[t]{0.5\textwidth}
\begin{center}
\includegraphics[height=3cm]{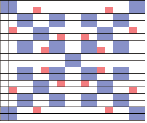}
\caption{$m=18$ }
\label{s18}
\end{center}
\end{minipage}
\end{tabular}
\end{figure}

\begin{figure}[htbp]
\begin{tabular}{cc}
\begin{minipage}[t]{0.5\textwidth}
\begin{center}
\includegraphics[height=3cm]{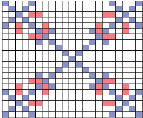}
\caption{ $m=21$}
\label{s21}
\end{center}
\end{minipage}
\begin{minipage}[t]{0.5\textwidth}
\begin{center}
\includegraphics[height=3cm]{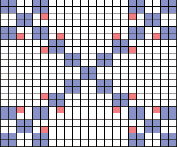}
\caption{$m=22$ }
\label{s22}
\end{center}
\end{minipage}
\end{tabular}
\end{figure}

\section*{Future Work}

As directions for future research, we are interested in whether there exist other combinatorial games whose sets of $P$-positions form fractal objects or their planar sections, and whether meaningful results can be obtained by varying the directions along which such self-similar structures are sliced.
The appearance of recursive and self-similar patterns in the chocolate game suggests that cellular automata provide a natural framework for describing and generating such configurations.
Motivated by this perspective, we would like to investigate what other classes of combinatorial games admit a generation of their $P$-positions via cellular automata.
In particular, it is of interest to ask whether there exist cellular automaton rules that generate the set of $P$-positions for the chocolate game $C_{m,n}$ in the case $m<n$.
Clarifying such rules may further illuminate the relationship between impartial games, self-similar structures, and discrete dynamical systems.

\section*{Acknowledgements}

This research was carried out with the cooperation of Yugo~Narisawa.
The authors are also deeply grateful to Ryohei~Miyadera, a teacher at Keimei Gakuin Junior and High School, for his invaluable assistance in the preparation of this manuscript.

\end{document}

%% file: figure01.pdf_tex
\begingroup%
  \makeatletter%
  \providecommand\color[2][]{%
    \errmessage{(Inkscape) Color is used for the text in Inkscape, but the package 'color.sty' is not loaded}%
    \renewcommand\color[2][]{}%
  }%
  \providecommand\transparent[1]{%
    \errmessage{(Inkscape) Transparency is used (non-zero) for the text in Inkscape, but the package 'transparent.sty' is not loaded}%
    \renewcommand\transparent[1]{}%
  }%
  \providecommand\rotatebox[2]{#2}%
  \newcommand*\fsize{\dimexpr\f@size pt\relax}%
  \newcommand*\lineheight[1]{\fontsize{\fsize}{#1\fsize}\selectfont}%
  \ifx\svgwidth\undefined%
    \setlength{\unitlength}{330.75bp}%
    \ifx\svgscale\undefined%
      \relax%
    \else%
      \setlength{\unitlength}{\unitlength * \real{\svgscale}}%
    \fi%
  \else%
    \setlength{\unitlength}{\svgwidth}%
  \fi%
  \global\let\svgwidth\undefined%
  \global\let\svgscale\undefined%
  \makeatother%
  \begin{picture}(1,0.46938776)%
    \lineheight{1}%
    \setlength\tabcolsep{0pt}%
    \put(0,0){\includegraphics[width=\unitlength,page=1]{figure01.pdf}}%
    \put(0.02040816,0.43367347){\color[rgb]{0.58823529,0.58823529,0.58823529}\makebox(0,0)[lt]{\lineheight{1.25}\smash{\begin{tabular}[t]{l}$m=1$\end{tabular}}}}%
    \put(0.02040816,0.43367347){\color[rgb]{0,0,0}\makebox(0,0)[lt]{\lineheight{1.25}\smash{\begin{tabular}[t]{l}$m=1$\end{tabular}}}}%
    \put(0,0){\includegraphics[width=\unitlength,page=2]{figure01.pdf}}%
    \put(0.14285714,0.43367347){\color[rgb]{0.58823529,0.58823529,0.58823529}\makebox(0,0)[lt]{\lineheight{1.25}\smash{\begin{tabular}[t]{l}$m=2$\end{tabular}}}}%
    \put(0.14285714,0.43367347){\color[rgb]{0,0,0}\makebox(0,0)[lt]{\lineheight{1.25}\smash{\begin{tabular}[t]{l}$m=2$\end{tabular}}}}%
    \put(0,0){\includegraphics[width=\unitlength,page=3]{figure01.pdf}}%
    \put(0.26530612,0.43367347){\color[rgb]{0.58823529,0.58823529,0.58823529}\makebox(0,0)[lt]{\lineheight{1.25}\smash{\begin{tabular}[t]{l}$m=3$\end{tabular}}}}%
    \put(0.26530612,0.43367347){\color[rgb]{0,0,0}\makebox(0,0)[lt]{\lineheight{1.25}\smash{\begin{tabular}[t]{l}$m=3$\end{tabular}}}}%
    \put(0,0){\includegraphics[width=\unitlength,page=4]{figure01.pdf}}%
    \put(0.3877551,0.43367347){\color[rgb]{0.58823529,0.58823529,0.58823529}\makebox(0,0)[lt]{\lineheight{1.25}\smash{\begin{tabular}[t]{l}$m=4$\end{tabular}}}}%
    \put(0.3877551,0.43367347){\color[rgb]{0,0,0}\makebox(0,0)[lt]{\lineheight{1.25}\smash{\begin{tabular}[t]{l}$m=4$\end{tabular}}}}%
    \put(0,0){\includegraphics[width=\unitlength,page=5]{figure01.pdf}}%
    \put(0.51020408,0.43367347){\color[rgb]{0.58823529,0.58823529,0.58823529}\makebox(0,0)[lt]{\lineheight{1.25}\smash{\begin{tabular}[t]{l}$m=5$\end{tabular}}}}%
    \put(0.51020408,0.43367347){\color[rgb]{0,0,0}\makebox(0,0)[lt]{\lineheight{1.25}\smash{\begin{tabular}[t]{l}$m=5$\end{tabular}}}}%
    \put(0,0){\includegraphics[width=\unitlength,page=6]{figure01.pdf}}%
    \put(0.63265306,0.43367347){\color[rgb]{0.58823529,0.58823529,0.58823529}\makebox(0,0)[lt]{\lineheight{1.25}\smash{\begin{tabular}[t]{l}$m=6$\end{tabular}}}}%
    \put(0.63265306,0.43367347){\color[rgb]{0,0,0}\makebox(0,0)[lt]{\lineheight{1.25}\smash{\begin{tabular}[t]{l}$m=6$\end{tabular}}}}%
    \put(0,0){\includegraphics[width=\unitlength,page=7]{figure01.pdf}}%
    \put(0.75510204,0.43367347){\color[rgb]{0.58823529,0.58823529,0.58823529}\makebox(0,0)[lt]{\lineheight{1.25}\smash{\begin{tabular}[t]{l}$m=7$\end{tabular}}}}%
    \put(0.75510204,0.43367347){\color[rgb]{0,0,0}\makebox(0,0)[lt]{\lineheight{1.25}\smash{\begin{tabular}[t]{l}$m=7$\end{tabular}}}}%
    \put(0,0){\includegraphics[width=\unitlength,page=8]{figure01.pdf}}%
    \put(0.87755102,0.43367347){\color[rgb]{0.58823529,0.58823529,0.58823529}\makebox(0,0)[lt]{\lineheight{1.25}\smash{\begin{tabular}[t]{l}$m=8$\end{tabular}}}}%
    \put(0.87755102,0.43367347){\color[rgb]{0,0,0}\makebox(0,0)[lt]{\lineheight{1.25}\smash{\begin{tabular}[t]{l}$m=8$\end{tabular}}}}%
    \put(0,0){\includegraphics[width=\unitlength,page=9]{figure01.pdf}}%
    \put(0.02040816,0.29081633){\color[rgb]{0.58823529,0.58823529,0.58823529}\makebox(0,0)[lt]{\lineheight{1.25}\smash{\begin{tabular}[t]{l}$m=9$\end{tabular}}}}%
    \put(0.02040816,0.29081633){\color[rgb]{0,0,0}\makebox(0,0)[lt]{\lineheight{1.25}\smash{\begin{tabular}[t]{l}$m=9$\end{tabular}}}}%
    \put(0,0){\includegraphics[width=\unitlength,page=10]{figure01.pdf}}%
    \put(0.14285714,0.29081633){\color[rgb]{0.58823529,0.58823529,0.58823529}\makebox(0,0)[lt]{\lineheight{1.25}\smash{\begin{tabular}[t]{l}$m=10$\end{tabular}}}}%
    \put(0.14285714,0.29081633){\color[rgb]{0,0,0}\makebox(0,0)[lt]{\lineheight{1.25}\smash{\begin{tabular}[t]{l}$m=10$\end{tabular}}}}%
    \put(0,0){\includegraphics[width=\unitlength,page=11]{figure01.pdf}}%
    \put(0.26530612,0.29081633){\color[rgb]{0.58823529,0.58823529,0.58823529}\makebox(0,0)[lt]{\lineheight{1.25}\smash{\begin{tabular}[t]{l}$m=11$\end{tabular}}}}%
    \put(0.26530612,0.29081633){\color[rgb]{0,0,0}\makebox(0,0)[lt]{\lineheight{1.25}\smash{\begin{tabular}[t]{l}$m=11$\end{tabular}}}}%
    \put(0,0){\includegraphics[width=\unitlength,page=12]{figure01.pdf}}%
    \put(0.3877551,0.29081633){\color[rgb]{0.58823529,0.58823529,0.58823529}\makebox(0,0)[lt]{\lineheight{1.25}\smash{\begin{tabular}[t]{l}$m=12$\end{tabular}}}}%
    \put(0.3877551,0.29081633){\color[rgb]{0,0,0}\makebox(0,0)[lt]{\lineheight{1.25}\smash{\begin{tabular}[t]{l}$m=12$\end{tabular}}}}%
    \put(0,0){\includegraphics[width=\unitlength,page=13]{figure01.pdf}}%
    \put(0.51020408,0.29081633){\color[rgb]{0.58823529,0.58823529,0.58823529}\makebox(0,0)[lt]{\lineheight{1.25}\smash{\begin{tabular}[t]{l}$m=13$\end{tabular}}}}%
    \put(0.51020408,0.29081633){\color[rgb]{0,0,0}\makebox(0,0)[lt]{\lineheight{1.25}\smash{\begin{tabular}[t]{l}$m=13$\end{tabular}}}}%
    \put(0,0){\includegraphics[width=\unitlength,page=14]{figure01.pdf}}%
    \put(0.63265306,0.29081633){\color[rgb]{0.58823529,0.58823529,0.58823529}\makebox(0,0)[lt]{\lineheight{1.25}\smash{\begin{tabular}[t]{l}$m=14$\end{tabular}}}}%
    \put(0.63265306,0.29081633){\color[rgb]{0,0,0}\makebox(0,0)[lt]{\lineheight{1.25}\smash{\begin{tabular}[t]{l}$m=14$\end{tabular}}}}%
    \put(0,0){\includegraphics[width=\unitlength,page=15]{figure01.pdf}}%
    \put(0.75510204,0.29081633){\color[rgb]{0.58823529,0.58823529,0.58823529}\makebox(0,0)[lt]{\lineheight{1.25}\smash{\begin{tabular}[t]{l}$m=15$\end{tabular}}}}%
    \put(0.75510204,0.29081633){\color[rgb]{0,0,0}\makebox(0,0)[lt]{\lineheight{1.25}\smash{\begin{tabular}[t]{l}$m=15$\end{tabular}}}}%
    \put(0,0){\includegraphics[width=\unitlength,page=16]{figure01.pdf}}%
    \put(0.87755102,0.29081633){\color[rgb]{0.58823529,0.58823529,0.58823529}\makebox(0,0)[lt]{\lineheight{1.25}\smash{\begin{tabular}[t]{l}$m=16$\end{tabular}}}}%
    \put(0.87755102,0.29081633){\color[rgb]{0,0,0}\makebox(0,0)[lt]{\lineheight{1.25}\smash{\begin{tabular}[t]{l}$m=16$\end{tabular}}}}%
    \put(0,0){\includegraphics[width=\unitlength,page=17]{figure01.pdf}}%
    \put(0.02040816,0.14795918){\color[rgb]{0.58823529,0.58823529,0.58823529}\makebox(0,0)[lt]{\lineheight{1.25}\smash{\begin{tabular}[t]{l}$m=17$\end{tabular}}}}%
    \put(0.02040816,0.14795918){\color[rgb]{0,0,0}\makebox(0,0)[lt]{\lineheight{1.25}\smash{\begin{tabular}[t]{l}$m=17$\end{tabular}}}}%
    \put(0,0){\includegraphics[width=\unitlength,page=18]{figure01.pdf}}%
    \put(0.14285714,0.14795918){\color[rgb]{0.58823529,0.58823529,0.58823529}\makebox(0,0)[lt]{\lineheight{1.25}\smash{\begin{tabular}[t]{l}$m=18$\end{tabular}}}}%
    \put(0.14285714,0.14795918){\color[rgb]{0,0,0}\makebox(0,0)[lt]{\lineheight{1.25}\smash{\begin{tabular}[t]{l}$m=18$\end{tabular}}}}%
    \put(0,0){\includegraphics[width=\unitlength,page=19]{figure01.pdf}}%
    \put(0.26530612,0.14795918){\color[rgb]{0.58823529,0.58823529,0.58823529}\makebox(0,0)[lt]{\lineheight{1.25}\smash{\begin{tabular}[t]{l}$m=19$\end{tabular}}}}%
    \put(0.26530612,0.14795918){\color[rgb]{0,0,0}\makebox(0,0)[lt]{\lineheight{1.25}\smash{\begin{tabular}[t]{l}$m=19$\end{tabular}}}}%
    \put(0,0){\includegraphics[width=\unitlength,page=20]{figure01.pdf}}%
    \put(0.3877551,0.14795918){\color[rgb]{0.58823529,0.58823529,0.58823529}\makebox(0,0)[lt]{\lineheight{1.25}\smash{\begin{tabular}[t]{l}$m=20$\end{tabular}}}}%
    \put(0.3877551,0.14795918){\color[rgb]{0,0,0}\makebox(0,0)[lt]{\lineheight{1.25}\smash{\begin{tabular}[t]{l}$m=20$\end{tabular}}}}%
    \put(0,0){\includegraphics[width=\unitlength,page=21]{figure01.pdf}}%
    \put(0.51020408,0.14795918){\color[rgb]{0.58823529,0.58823529,0.58823529}\makebox(0,0)[lt]{\lineheight{1.25}\smash{\begin{tabular}[t]{l}$m=21$\end{tabular}}}}%
    \put(0.51020408,0.14795918){\color[rgb]{0,0,0}\makebox(0,0)[lt]{\lineheight{1.25}\smash{\begin{tabular}[t]{l}$m=21$\end{tabular}}}}%
    \put(0,0){\includegraphics[width=\unitlength,page=22]{figure01.pdf}}%
    \put(0.63265306,0.14795918){\color[rgb]{0.58823529,0.58823529,0.58823529}\makebox(0,0)[lt]{\lineheight{1.25}\smash{\begin{tabular}[t]{l}$m=22$\end{tabular}}}}%
    \put(0.63265306,0.14795918){\color[rgb]{0,0,0}\makebox(0,0)[lt]{\lineheight{1.25}\smash{\begin{tabular}[t]{l}$m=22$\end{tabular}}}}%
    \put(0,0){\includegraphics[width=\unitlength,page=23]{figure01.pdf}}%
    \put(0.75510204,0.14795918){\color[rgb]{0.58823529,0.58823529,0.58823529}\makebox(0,0)[lt]{\lineheight{1.25}\smash{\begin{tabular}[t]{l}$m=23$\end{tabular}}}}%
    \put(0.75510204,0.14795918){\color[rgb]{0,0,0}\makebox(0,0)[lt]{\lineheight{1.25}\smash{\begin{tabular}[t]{l}$m=23$\end{tabular}}}}%
    \put(0,0){\includegraphics[width=\unitlength,page=24]{figure01.pdf}}%
    \put(0.87755102,0.14795918){\color[rgb]{0.58823529,0.58823529,0.58823529}\makebox(0,0)[lt]{\lineheight{1.25}\smash{\begin{tabular}[t]{l}$m=24$\end{tabular}}}}%
    \put(0.87755102,0.14795918){\color[rgb]{0,0,0}\makebox(0,0)[lt]{\lineheight{1.25}\smash{\begin{tabular}[t]{l}$m=24$\end{tabular}}}}%
    \put(0,0){\includegraphics[width=\unitlength,page=25]{figure01.pdf}}%
  \end{picture}%
\endgroup%

%% file: figure02.pdf_tex
\begingroup%
  \makeatletter%
  \providecommand\color[2][]{%
    \errmessage{(Inkscape) Color is used for the text in Inkscape, but the package 'color.sty' is not loaded}%
    \renewcommand\color[2][]{}%
  }%
  \providecommand\transparent[1]{%
    \errmessage{(Inkscape) Transparency is used (non-zero) for the text in Inkscape, but the package 'transparent.sty' is not loaded}%
    \renewcommand\transparent[1]{}%
  }%
  \providecommand\rotatebox[2]{#2}%
  \newcommand*\fsize{\dimexpr\f@size pt\relax}%
  \newcommand*\lineheight[1]{\fontsize{\fsize}{#1\fsize}\selectfont}%
  \ifx\svgwidth\undefined%
    \setlength{\unitlength}{326.25bp}%
    \ifx\svgscale\undefined%
      \relax%
    \else%
      \setlength{\unitlength}{\unitlength * \real{\svgscale}}%
    \fi%
  \else%
    \setlength{\unitlength}{\svgwidth}%
  \fi%
  \global\let\svgwidth\undefined%
  \global\let\svgscale\undefined%
  \makeatother%
  \begin{picture}(1,0.44827586)%
    \lineheight{1}%
    \setlength\tabcolsep{0pt}%
    \put(0,0){\includegraphics[width=\unitlength,page=1]{figure02.pdf}}%
    \put(0.03448276,0.42528736){\color[rgb]{0.58823529,0.58823529,0.58823529}\makebox(0,0)[lt]{\lineheight{1.25}\smash{\begin{tabular}[t]{l}$R_{3,3}$\end{tabular}}}}%
    \put(0.03448276,0.42528736){\color[rgb]{0,0,0}\makebox(0,0)[lt]{\lineheight{1.25}\smash{\begin{tabular}[t]{l}$R_{3,3}$\end{tabular}}}}%
    \put(0,0){\includegraphics[width=\unitlength,page=2]{figure02.pdf}}%
    \put(0.03448276,0.14942529){\color[rgb]{0.58823529,0.58823529,0.58823529}\makebox(0,0)[lt]{\lineheight{1.25}\smash{\begin{tabular}[t]{l}$R_{3,3}$\end{tabular}}}}%
    \put(0.03448276,0.14942529){\color[rgb]{0,0,0}\makebox(0,0)[lt]{\lineheight{1.25}\smash{\begin{tabular}[t]{l}$R_{3,3}$\end{tabular}}}}%
    \put(0,0){\includegraphics[width=\unitlength,page=3]{figure02.pdf}}%
    \put(0.17241379,0.32183908){\color[rgb]{0.58823529,0.58823529,0.58823529}\makebox(0,0)[lt]{\lineheight{1.25}\smash{\begin{tabular}[t]{l}$R_{5,5}$\end{tabular}}}}%
    \put(0.17241379,0.32183908){\color[rgb]{0,0,0}\makebox(0,0)[lt]{\lineheight{1.25}\smash{\begin{tabular}[t]{l}$R_{5,5}$\end{tabular}}}}%
    \put(0,0){\includegraphics[width=\unitlength,page=4]{figure02.pdf}}%
    \put(0.37931034,0.42528736){\color[rgb]{0.58823529,0.58823529,0.58823529}\makebox(0,0)[lt]{\lineheight{1.25}\smash{\begin{tabular}[t]{l}$R_{3,3}$\end{tabular}}}}%
    \put(0.37931034,0.42528736){\color[rgb]{0,0,0}\makebox(0,0)[lt]{\lineheight{1.25}\smash{\begin{tabular}[t]{l}$R_{3,3}$\end{tabular}}}}%
    \put(0,0){\includegraphics[width=\unitlength,page=5]{figure02.pdf}}%
    \put(0.37931034,0.14942529){\color[rgb]{0.58823529,0.58823529,0.58823529}\makebox(0,0)[lt]{\lineheight{1.25}\smash{\begin{tabular}[t]{l}$R_{3,3}$\end{tabular}}}}%
    \put(0.37931034,0.14942529){\color[rgb]{0,0,0}\makebox(0,0)[lt]{\lineheight{1.25}\smash{\begin{tabular}[t]{l}$R_{3,3}$\end{tabular}}}}%
    \put(0,0){\includegraphics[width=\unitlength,page=6]{figure02.pdf}}%
    \put(0.5862069,0.42528736){\color[rgb]{0.58823529,0.58823529,0.58823529}\makebox(0,0)[lt]{\lineheight{1.25}\smash{\begin{tabular}[t]{l}$R_{11,11}$\end{tabular}}}}%
    \put(0.5862069,0.42528736){\color[rgb]{0,0,0}\makebox(0,0)[lt]{\lineheight{1.25}\smash{\begin{tabular}[t]{l}$R_{11,11}$\end{tabular}}}}%
    \put(0,0){\includegraphics[width=\unitlength,page=7]{figure02.pdf}}%
  \end{picture}%
\endgroup%

%% file: figure05.pdf_tex
\begingroup%
  \makeatletter%
  \providecommand\color[2][]{%
    \errmessage{(Inkscape) Color is used for the text in Inkscape, but the package 'color.sty' is not loaded}%
    \renewcommand\color[2][]{}%
  }%
  \providecommand\transparent[1]{%
    \errmessage{(Inkscape) Transparency is used (non-zero) for the text in Inkscape, but the package 'transparent.sty' is not loaded}%
    \renewcommand\transparent[1]{}%
  }%
  \providecommand\rotatebox[2]{#2}%
  \newcommand*\fsize{\dimexpr\f@size pt\relax}%
  \newcommand*\lineheight[1]{\fontsize{\fsize}{#1\fsize}\selectfont}%
  \ifx\svgwidth\undefined%
    \setlength{\unitlength}{337.5bp}%
    \ifx\svgscale\undefined%
      \relax%
    \else%
      \setlength{\unitlength}{\unitlength * \real{\svgscale}}%
    \fi%
  \else%
    \setlength{\unitlength}{\svgwidth}%
  \fi%
  \global\let\svgwidth\undefined%
  \global\let\svgscale\undefined%
  \makeatother%
  \begin{picture}(1,0.35555556)%
    \lineheight{1}%
    \setlength\tabcolsep{0pt}%
    \put(0,0){\includegraphics[width=\unitlength,page=1]{figure05.pdf}}%
    \put(0.04444444,0.28148148){\color[rgb]{0.58823529,0.58823529,0.58823529}\makebox(0,0)[lt]{\lineheight{1.25}\smash{\begin{tabular}[t]{l}$H_{n,5}$\end{tabular}}}}%
    \put(0.04444444,0.28148148){\color[rgb]{0,0,0}\makebox(0,0)[lt]{\lineheight{1.25}\smash{\begin{tabular}[t]{l}$H_{n,5}$\end{tabular}}}}%
    \put(0,0){\includegraphics[width=\unitlength,page=2]{figure05.pdf}}%
    \put(0.35555556,0.29259259){\color[rgb]{0.58823529,0.58823529,0.58823529}\makebox(0,0)[lt]{\lineheight{1.25}\smash{\begin{tabular}[t]{l}$H_{n,5.5}$\end{tabular}}}}%
    \put(0.35555556,0.29259259){\color[rgb]{0,0,0}\makebox(0,0)[lt]{\lineheight{1.25}\smash{\begin{tabular}[t]{l}$H_{n,5.5}$\end{tabular}}}}%
    \put(0,0){\includegraphics[width=\unitlength,page=3]{figure05.pdf}}%
    \put(0.68888889,0.3037037){\color[rgb]{0.58823529,0.58823529,0.58823529}\makebox(0,0)[lt]{\lineheight{1.25}\smash{\begin{tabular}[t]{l}$H_{n,6}$\end{tabular}}}}%
    \put(0.68888889,0.3037037){\color[rgb]{0,0,0}\makebox(0,0)[lt]{\lineheight{1.25}\smash{\begin{tabular}[t]{l}$H_{n,6}$\end{tabular}}}}%
    \put(0,0){\includegraphics[width=\unitlength,page=4]{figure05.pdf}}%
  \end{picture}%
\endgroup%